\documentclass[11pt]{amsart}

\usepackage{amssymb}
\usepackage{amsfonts}
\usepackage{amsmath}
\usepackage{color}

\usepackage{color}
\usepackage{amsxtra}
\usepackage{mathtools}
\usepackage{enumerate}
\usepackage{blindtext}
\usepackage[inline]{enumitem}
\usepackage{nicefrac}
\mathtoolsset{showonlyrefs} %	enumera solo las ecuaciones que se citan
\usepackage{comment} %para comentar
\usepackage{mathrsfs}
\usepackage{tikz}

\usetikzlibrary{shadings}
\usepackage{color}

\usepackage{hyperref}
\newtheorem{theorem}{Theorem}[section]

\newtheorem{definition}[theorem]{Definition}

\newtheorem{lemma}[theorem]{Lemma}

\theoremstyle{remark}
\newtheorem{remark}[theorem]{Remark}

\numberwithin{equation}{section}

\def\RR{{\mathbb R} }

\begin{document}

%%%%%%%%%%%%%%%%%%%%%%%%%%%%%%%%%%%%%%%%%%%%%%%%%%%%
%Title
%%%%%%%%%%%%%%%%%%%%%%%%%%%%%%%%%%%%%%%%%%%%%%%%%%%%
\title[The $p-$Laplacian equation with dynamical boundary conditions]{THE LIMIT AS $p\rightarrow\infty$ FOR THE $p-$LAPLACIAN EQUATION WITH DYNAMICAL BOUNDARY CONDITIONS}
%%%%%%%%%%%%%%%%%%%%%%%%%%%%%%%%%%%%%%%%%%%%%%%%%%%%%%%%%%%%%%%%%%%%%%%%%%%
%Authors
%%%%%%%%%%%%%%%%%%%%%%%%%%%%%%%%%%%%%%%%%%%%%%%%%%%%%%%%%%%%%%%%%%%%%%%%%%%

\author[E. \"{O}zt\"{u}rk]{Eylem \"{O}zt\"{u}rk}
\address{Eylem \"{O}zt\"{u}rk \hfill\break\indent 
Department of Mathematics, 
\hfill\break\indent Hacettepe University,
\hfill\break\indent 06800 Beytepe, 
\hfill\break\indent Ankara, Turkey. }
\email{eyturk1983@gmail.com}
    
\author[J. D. Rossi]{Julio D. Rossi}
	\address{Julio D. Rossi
	\hfill\break\indent
        Departamento  de Matem{\'a}tica, FCEyN,\hfill\break\indent
        Universidad de Buenos Aires,
        \hfill\break\indent Pabellon I,
         Ciudad Universitaria (C1428BCW),
         \hfill\break\indent Buenos Aires, Argentina.}
\email{jrossi@dm.uba.ar}

\begin{abstract} In this paper we study the limit as $p\to \infty$ in the evolution problem
driven by the $p-$Laplacian with dynamical boundary conditions. We prove that the natural
energy functional associated with this problem converges to a limit in the sense of Mosco convergence 
and as a consequence we obtain convergence of the solutions to the evolution problems. 
For the limit problem we show an interpretation in terms of optimal mass transportation and provide
examples of explicit solutions for some particular data. 

\smallskip

\noindent \text{Keywords.} {$p-$Laplacian, Dynamical boundary conditions, Mosco convergence.}

\smallskip

%35K20  	Initial-boundary value problems for second-order parabolic equations
%35K55 Nonlinear parabolic equations
%35K92  	Quasilinear parabolic equations with $p$-Laplacian
%47J35  	Nonlinear evolution equations
\end{abstract}

\maketitle

\centerline{\it To the memory of Alan Lazer, a great mathematician.}

\section{Introduction}
Our main purpose in this paper is to study a nonlinear diffusion equation
obtained as the limit as $p\rightarrow\infty$ to the $p-$Laplacian with dynamical boundary conditions.
More precisely, we look for the limit as $p\rightarrow\infty$ of the solutions to the following problem;
\begin{equation} \label{p-lapla-intro}  
\begin{cases}
\displaystyle 0= \Delta_p u (x,t), & x\in\Omega, \, t>0,\\[7pt]
\displaystyle \frac{\partial u}{\partial t}(x,t) + |\nabla u|^{p-2} \frac{\partial u}{\partial \eta} (x,t) = f(x,t), & x\in {\partial \Omega}, \, t>0,\\[7pt]
u(x,0)=u_0(x),  & x\in{\partial \Omega}.
\end{cases}
\end{equation}
Here $\Omega \subset \mathbb{R}^N$ is a bounded smooth domain, $\frac{\partial u}{\partial \eta}$ denotes the outer normal derivative of $u$ and 
$f$ is a nonnegative function that represents a given source term localized on $\partial \Omega$, 
which is interpreted physically as adding material to an evolving system, within which mass particles are continually rearranged by diffusion
( in this kind of model it is assumed that diffusion is much faster inside the domain than on the boundary, hence
the time derivative appears only in the boundary condition).

Associated with this evolution problem we have the following functional
$E_p : L^2 (\partial \Omega) \mapsto \mathbb{R} \cup \{+\infty\}$,
\begin{equation} \label{Ep}
E_p (u) = \left\{
\begin{array}{ll}
\displaystyle \min_{v \in W^{1,p} (\Omega), trace (v)=u} \frac1p \int_\Omega |\nabla v|^p  \qquad & u \in trace (W^{1,p} (\Omega)), \\[7pt]
+\infty \qquad & u \not\in trace (W^{1,p} (\Omega)).
\end{array}
\right.
\end{equation}

As we mentioned before our aim is to look for the limit as $p\to \infty$ of the solutions $u_p$ to \eqref{p-lapla-intro}.
To this end we use a general result by U. Mosco, see \cite{Mosco1,Mosco2}: if the associated functionals converge to a
limit functional
(in an adequate sense, that roughly speaking, means convergence of the epigraphs, see Section \ref{sect-prelim} for the precise definition) 
then the corresponding solutions
to the associated evolution problems converge to the solution associated with the limit functional. 

The limit of the functionals $E_p$ as $p\to \infty$ is given by
$E_\infty : L^2 (\partial \Omega) \mapsto \mathbb{R} \cup \{+\infty\}$,
\begin{equation} \label{Einfty}
E_\infty (u) = \left\{
\begin{array}{ll}
\displaystyle 0  \qquad & u \in A_\infty, \\[7pt]
+\infty \qquad & u \not\in A_\infty,
\end{array}
\right.
\end{equation}
with
$$
A_\infty = \Big\{ u \in C(\partial \Omega) : \exists v :\overline{\Omega} \mapsto \mathbb{R} \mbox{ with } |\nabla v|\leq 1 \mbox{a.e } \Omega,
v|_{\partial \Omega} = u \Big\}.
$$

Our first result is given by.

\begin{theorem} \label{teo.Mosco}
The functionals $E_p$ converge to $E_\infty$ as $p\to \infty$ in the Mosco sense. 
\end{theorem}

As a consequence we have the convergence of the solutions to 
our evolution problem \eqref{p-lapla-intro}.

\begin{theorem} \label{corol.converg}
Let $u_p(x,t)$ be the solution of the problem \eqref{p-lapla-intro} with a fixed initial condition $u_0 \in \overline{A_\infty}^{L^2 (\partial \Omega)}$ 
and a fixed right hand side $f \in L^1 (0,T: L^2 (\partial \Omega))$.

Then, 
\begin{equation} \label{eq.conver}
u_p \to u_\infty
\end{equation}
as $p\to \infty$ in $C([0,T]: L^2(\partial \Omega))$, that is,
$$
\lim_{p\to \infty } \max_{t\in [0,T]} \| u_p(\cdot,t) - u_\infty (\cdot,t) \|_{L^2(\partial \Omega)} = 0.
$$
Moreover, the limit $u_\infty$ is characterized as the solution to 
\begin{equation}\label{eq-limite}
\left\{
\begin{array}{ll}
\displaystyle f(x,t)- \frac{\partial u}{\partial t} (x,t) \in \partial E_\infty (u(x,t))  \qquad & x\in \partial \Omega, t>0, \\[7pt]
u(x,0)=u_0(x) \qquad & x \in \partial \Omega.
\end{array}
\right.
\end{equation}

If we assume that $u_0 \in L^1 (\partial \Omega)$ and $f$ is such that
$$
\sup_{t \in [0,T]} \int_{\partial \Omega} |f (x,t) | d\sigma (x) +  \int_{\partial \Omega} \Big| \frac{\partial f}{\partial t} (x,t)\Big| d\sigma (x) < +\infty.  
$$
Then, there exists a subsequence $p_i \to \infty$  such that 
\begin{equation} \label{convergencias}
\begin{array}{l}
\displaystyle 
u_{p_i} \to u_\infty \qquad \mbox{ a.e. and strongly in } L^2(\partial \Omega \times [0, T]), \\[7pt]
\displaystyle \nabla u_{p_i} \rightharpoonup \nabla u_\infty \qquad \mbox{ weakly in } L^2(\partial \Omega \times [0, T]), \\[7pt] 
\displaystyle \frac{\partial u_{p_i}}{\partial t} \rightharpoonup \frac{\partial u_\infty}{\partial t} \qquad \mbox{ weakly in } L^2(\partial \Omega \times [0, T]).
\end{array}
\end{equation}
\end{theorem}

Finally, we relate the limit problem with an optimal mass transport problem in Theorem \ref{teo-mass}. the optimal mass transport problem is
defined with a cost given by the distance between points on $\partial \Omega$ considering paths inside $\Omega$, that is defined as the minimum of the lengths 
of the paths inside $\Omega$ that join the two points. We call this distance $d_\Omega$.   
It turns out that the limit of the solution to the limit problem, $u_\infty(\cdot,t)$, is a Kantorovich potential for
the optimal mass transport problem between $f(\cdot,t)$ and $\frac{\partial u_\infty}{\partial t} (\cdot,t)$. 

\begin{theorem} \label{teo-mass}
The solution to the limit problem \eqref{eq-limite} satisfies
$$
\begin{array}{l}
\displaystyle 
\int_{\partial \Omega } 
u_\infty (x,t) \Big(\frac{\partial u_\infty }{\partial t}(x,t) - f(x,t) \Big) d\sigma (x) \\[7pt]
\qquad \displaystyle = \max_{v: | v(x)-v(y)| \leq d_\Omega (x,y)} 
\int_{\partial \Omega } 
v(x) \Big(\frac{\partial u_\infty }{\partial t}(x,t) - f(x,t)\Big) d\sigma (x),
\end{array}
$$
that is, $u_\infty $ is a Kantorovich potential for the dual formulation of the Monge-Kantorovich mass transport problem 
between $f(\cdot,t) d\sigma $ and $\frac{\partial u_\infty }{\partial t} (\cdot,t) d\sigma $.
\end{theorem}

Therefore, as was pointed out in \cite{AEW}, the limit problem  \eqref{eq-limite} 
can be interpreted as a model for the formation and growth of a sandpile where particles of sand are distributed on $\partial \Omega$
(here $u_\infty (x,t)$ describes the amount of the sand at the point $x$ at time $t$). 
The main assumption being that the sandpile is stable when the slope is less than or equal to one and unstable if not.

We also include some explicit examples of solutions to the limit problem. In these examples 
one can appreciate the mass transport interpretation of the limit problem. Also we illustrate
a curious phenomenon, the support of the solution on $\partial \Omega$ may be disconnected even
if the domain is strictly convex, the initial condition is zero and the reaction has connected support.

Dynamical boundary conditions appear in modeling physical phenomena when there is a thin layer 
around the boundary in which reaction takes place. We refer to \cite{11,12,13,15,Latorre,Marcone,27} for general references concerning 
evolution problems with this kind of boundary conditions.

As a precedent concerning limits as $p\to \infty$, we mention that the problem \eqref{p-lapla-intro} in the elliptic (time independent) case
was studied in \cite{gampr} (see also \cite{gampr2} for the associated eigenvalue problem). 
Here one needs to assume that
$\int_{\partial \Omega} f =0$
(otherwise, there is no solution) and in order to have uniqueness of solutions one normalizes according to 
$\int_{\partial \Omega}u =0$. 

Concerning evolution problems with the $p-$Laplacian, the counterpart of our results for the Cauchy problem
was obtained 
in \cite{AEW} and \cite{EFG}. In those references it was studied the limiting behavior as $p\to \infty$ 
of solutions to the quasilinear parabolic problem
$$
\left\{
\begin{array}{ll}
\displaystyle \frac{\partial v}{\partial t} (x,t) - \Delta_p v (x,t) = f (x,t), \quad & \mbox{in }(0,T)\times \mathbb{R}^N, \\[7pt]
v(x,0) = u_0(x), \quad & \mbox{ in }  \mathbb{R}^N.
\end{array} \right.
$$
In \cite{AEW}, assuming that $u_0$ is a Lipschitz function with compact support, satisfying $|\nabla u_0| \leq 1$,
it is proved that $v_p \to v_\infty$ and the limit function $v_\infty$ satisfies
$$
f (x,t) - \frac{\partial v_\infty}{\partial t} (x,t) \in \partial F_\infty (v_\infty (x,t)),
$$
with
$$
F_\infty (v)=
\left\{
\begin{array}{ll} 0, \qquad &\mbox{ if } |\nabla v|\leq 1, \\[7pt]
+\infty , \qquad & \mbox{ in other case.}
\end{array}
\right.
$$

Other related papers that deal with limits as $p\to \infty$ in $p-$Laplacian
problems are \cite{BBM,BBP,Bocea,HL}. The relation between a limit as $p\to \infty$ in a $p-$Laplacian problem and optimal mass transport 
was first found in \cite{EG} (see also \cite{BBP}).

\medskip

The rest of the paper is organized as follows: in Section \ref{sect-prelim} we gather some preliminary results 
concerning Mosco convergence of functionals; in Section \ref{sect-proof} we prove the convergence of the functionals $E_p$ to $E_\infty$ stated
Theorem \ref{teo.Mosco} and we deduce the convergence of the solutions to the evolution problems in Theorem \ref{corol.converg}.
In Section \ref{sect-Mass} we deal with the Mass transport interpretation of the limit problem.
Finally, in Section \ref{sect-examples} we include some explicit examples of solutions to the limit problem.

\section{Preliminaries} \label{sect-prelim}

Next, we recall the definition of Mosco-convergence. If $X$ is a
metric space, and $\{A_n\}$ is a sequence of subsets of $X$, we
define
$$\liminf_{n \to \infty} A_n := \Big\{ x \in X \ : \
\exists x_n \in A_n, \ x_n \to x \Big\},$$
and $$
\limsup_{n \to
\infty} A_n := \Big\{ x \in X \ : \ \exists x_{n_k} \in A_{n_k}, \
x_{n_k} \to x \Big\}.$$ 
If $X$ is a normed space, we denote by $s-\lim$
and $w-\lim$ the above limits associated, respectively, to the
strong and to the weak topology of $X$.

\begin{definition}\label{d1}
Let $H$ be a Hilbert space. Given $\Psi_n, \Psi : H \rightarrow (-
\infty, + \infty]$ convex, lower-semicontinuous functionals, we say
that $\Psi_n$ converges to $\Psi$ in the sense of Mosco if
\begin{equation}\label{e1Mos.prelim}
w-\limsup_{n \to \infty} {\rm Epi}(\Psi_n) \subset {\rm Epi}(\Psi)
\subset s-\liminf_{n \to \infty}{\rm Epi}(\Psi_n),
\end{equation}
where ${\rm Epi} (\Psi_n)$ and ${\rm Epi} (\Psi)$ denote the
epigraphs of the functionals $\Psi_n$ and $\Psi$, defined by
$${\rm Epi}(\Psi_n):=\Big\{(u,\lambda)\in L^2(\RR^N)\times\RR \ : \
\lambda\geq \Psi_n(u)\Big\}, $$ 
and $$
{\rm Epi}(\Psi)
:=\Big\{(u,\lambda)\in L^2(\RR^N)\times\RR \ : \ \lambda\geq \Psi(u)\Big\}\,.$$
\end{definition}

\begin{remark} {\rm
We note that (\ref{e1Mos.prelim}) is equivalent to the requirement
that the following two conditions are simultaneously satisfied:
\begin{equation}\label{e2Mos.prelim}
\forall \, u \in D(\Psi) \ \ \exists u_n \in D(\Psi_n) \ : \ u_n \to u \ \ \hbox{and} \ \
\Psi(u) \geq \limsup_{n \to \infty} \Psi_n(u_n);
\end{equation}
\begin{equation}\label{e3Mos.prelim}
\hbox{for every subsequence} \ \{n_k\}, \ \Psi(u) \leq \liminf_{k}
\Psi_{n_k}(u_k) \ \hbox{whenever} \ u_{k}\rightharpoonup u.
\end{equation}
Here $D(\Psi) := \{ u \in H \ : \ \Psi (u) < \infty \}$ and
$D(\Psi_n) := \{ u \in H \ : \ \Psi_n (u) < \infty \}$ denote the
domains of $\Psi$ and $\Psi_n$, respectively.}
\end{remark}

To identify the limit of the solutions $u_n$ to problem
\eqref{p-lapla-intro} (see the Introduction), we will use
the methods of Convex Analysis, and so we must first recall some
terminology (see
\cite{ET},
\cite{Brezis} and \cite{Attouch}).

If $H$ is a real Hilbert space with inner product $( \cdot , \cdot)$ 
and $\Psi : H \rightarrow (- \infty, + \infty]$ is convex, then
the subdifferential of $\Psi$ is defined as the multivalued operator
$\partial \Psi$ given by
$$v \in \partial \Psi(u) \ \iff \ \Psi(w) - \Psi(u) \geq (v, w -u)
\ \ \ \forall \, w \in H.$$
Recall that the epigraph of $\Psi$ is defined by
$${\rm Epi}(\Psi) = \Big\{ (u, \lambda) \in H \times \RR \ : \  \lambda \geq \Psi(u) \Big\}.$$
Given $K$  a closed convex subset of $H$, we define the indicator
function of $K$ by
$$
I_K(u) = \left\{ \begin{array}{ll} 0 \qquad & \hbox{if} \ \ u \in
K
,\\[7pt]
+ \infty
\qquad & \hbox{if} \ \ u \not\in K. \end{array}\right.
$$
Then the subdifferential is characterized by
\begin{equation}\nonumber\label{e1Mazo}v \in \partial I_K(u) \
\iff \ u \in K \ \ \hbox{and} \ \ (v, w - u ) \leq 0 \ \ \ \forall \, w \in K.
\end{equation}

When the convex functional  $\Psi : H \rightarrow (- \infty, +
\infty]$ is proper, lower-semicontinuous, and such that $\min
\Psi=0$, it is well known (see \cite{Brezis}) that the abstract
Cauchy problem
\begin{equation}\nonumber\label{ACP}
\left\{\begin{array}{ll} u_t + \partial \Psi (u) \ni
f, \quad &\hbox{ a.e} \ t \in (0,T),
\\[7pt]
u(0)=u_0,&
\end{array}\right.
\end{equation}
has a unique solution for any $f \in L^1(0,T;H)$ and $u_0 \in
\overline{D(\partial \Psi)}$.

The Mosco convergence is a very useful tool to study convergence of
solutions of parabolic problems. The following theorem is a
consequence of results in \cite{BP} and \cite{Attouch}.

\begin{theorem}\label{convergencia1.Mosco} Let $\Psi_n, \Psi : H
\rightarrow (- \infty, + \infty]$  be convex and lower
semicontinuous functionals. Then the following statements are
equivalent:
\begin{itemize}
\item[(i)] $\Psi_n$ converges to $\Psi$ in the sense
of Mosco.

\medskip

\item[(ii)] $(I + \lambda \partial \Psi_n)^{-1} u \to (I + \lambda \partial \Psi)^{-1}
u$ \ \ $\forall \, \lambda > 0, \ u \in H$.
\end{itemize}
Moreover, either one of the above conditions, $(i)$ or $(ii)$, imply
that
\begin{itemize}
\item[(iii)] for every $u_0 \in \overline{D(\partial \Psi)}$ and
$u_{0,n} \in \overline{D(\partial \Psi_n)}$ such that $u_{0,n} \to
u_0$, and for every $f_n, f \in L^1 (0,T; H)$ with $f_n \to f$, if
$u_n(t)$, $u(t)$ are solutions of the abstract Cauchy problems
$$
\left\{\begin{array}{ll} (u_n)_t + \partial \Psi_n (u_n)
\ni f_n \quad &\hbox{ a.e.} \ \ t \in (0,T)
\\[7pt]
u_{n}(0)=u_{0,n},&
\end{array}\right.  $$
and
$$
\left\{\begin{array}{ll} u_t + \partial \Psi (u) \ni f
\quad &\hbox{ a.e.} \ \ t \in (0,T)
\\[7pt]
u(0)=u_0,&
\end{array}\right.  $$
respectively, then $$u_n \to u \qquad \mbox{ in } C([0,T]: H).$$
\end{itemize}
\end{theorem}

\section{Mosco convergence of the functionals and convergence of the solutions} \label{sect-proof}

First, we show some uniform bounds (independent of $p$) for the solutions $u_p$ to \eqref{p-lapla-intro}.

\begin{lemma} \label{lema.cotas} Fix $T>0$. Assume that $u_0 \in L^1 (\partial \Omega)$ and $f$ is such that
\begin{equation} \label{Cf}
C(f) := \sup_{t \in [0,T]} \int_{\partial \Omega} |f (x,t)| d\sigma(x) 
+  \int_{\partial \Omega} \Big| \frac{\partial f}{\partial t} (x,t) \Big| d\sigma(x) < +\infty.  
\end{equation}

Then, there exists a constant $C$ such that
\begin{equation} \label{cotas.eq}
\begin{array}{l}
\displaystyle 
\sup_{\partial \Omega \times [0,T]} |u_p | \leq C, \\[7pt]
\displaystyle \int_0^T \int_{\partial \Omega} \Big|\frac{\partial u_p}{\partial t} \Big|^2 \leq C, \\[7pt] 
\displaystyle \left(\int_0^T \int_{\partial \Omega}|\nabla u_p |^p \right)^{1/p} \leq C^{1/p},
\end{array}
\end{equation}
for every $N+1 \leq p <\infty$. The constant $C$ depends on $u_0$, $C(f)$ and $T$.
\end{lemma}

\begin{proof} Along this proof we denote by ${C}$ a generic constant that depends only on $u_0$, $C(f)$ and $T$
and may change from one line to another.

Now, we argue with the weak form of \eqref{p-lapla-intro}. It holds that
$$
\int_0^t \int_{\partial \Omega} \frac{\partial u_p}{\partial t} v + \int_0^t \int_\Omega
|\nabla u_p|^{p-2} \nabla u_p \nabla v  =  \int_0^t \int_{\partial \Omega} f v.
$$
Choose a smooth, nondecreasing function $\beta : \mathbb{R} \mapsto \mathbb{R}$ 
such that $\beta (x)=sgn(x)$ for $|x| \geq \delta >0$. By approximation we 
set $v=\beta (u_p)$ as the test function in the weak form of \eqref{p-lapla-intro}, to obtain
$$
\int_0^t \int_{\partial \Omega} \frac{\partial u_p}{\partial t} \beta (u_p)   \leq  \int_0^t \int_{\partial \Omega} f \beta (u_p).
$$
Hence, we get, 
$$
\int_{\partial \Omega} B(u_p) (t)  - \int_{\partial \Omega}  B(u_0)   =
\int_0^t \int_{\partial \Omega} \frac{\partial B(u_p)}{\partial t}   \leq  \int_0^t \int_{\partial \Omega} f \beta (u_p),
$$
here $B$ satisfies $B' (s)=\beta(s)$. Letting $\delta \to 0$ we obtain 
$$
\sup_{t \in [0,T]}\int_{\partial \Omega} |u_p| (t)  \leq \int_{\partial \Omega}  |u_0| + \int_0^T \int_{\partial \Omega} |f|
\leq \|u_0\|_{L^1 (\partial \Omega)} + C(f) T ,
$$
where $C(f)$ is the constant that depends on $f$ given in \eqref{Cf}. 

Now, if we take $v=u_p$ as a test function we get
$$
\int_0^t \int_{\partial \Omega} \frac{\partial u_p}{\partial t} u_p + \int_0^t \int_\Omega
|\nabla u_p|^{p}  =  \int_0^t \int_{\partial \Omega} f u_p.
$$
Since
$$
  \int_0^t \int_{\partial \Omega} f u_p 
  \leq C(f) \int_0^T \|u_p \|_{L^\infty (\partial \Omega)} ,
  $$
we obtain
\begin{equation} \label{ooo}
\frac12 \int_{\partial \Omega}| u_p|^2 (t) + \int_0^t \int_\Omega
|\nabla u_p|^{p}  \leq \frac12 \int_{\partial \Omega}| u_0|^2 +  C(f) \int_0^T \|u_p \|_{L^\infty (\partial \Omega)} .
\end{equation}
Hence, 
\begin{equation} \label{ooopp}
\begin{array}{l}
\displaystyle 
\sup_{t \in [0,T]} \frac12 \int_{\partial \Omega}| u_p|^2 (t) + \int_0^T \int_\Omega
|\nabla u_p|^{p} \\[7pt]
\qquad \displaystyle  \leq \frac12 \int_{\partial \Omega}| u_0|^2 +  C(f) \int_0^T \|u_p \|_{L^\infty (\partial \Omega)} .
\end{array}
\end{equation}

Since $u_p$ belongs to $W^{1,p} (\Omega)$, for $p \geq N+1$ we have
$$
\begin{array}{l}
\displaystyle 
\|u_p (t) \|_{L^\infty (\partial \Omega)}  \leq C \Big\{ \|\nabla u_p (t) \|_{L^{N+1} (\Omega)} + \| u_p (t) \|_{L^1(\partial \Omega)} \Big\}
\\[7pt]
\displaystyle \qquad \qquad \qquad 
\leq C \Big\{ \|\nabla u_p (t) \|_{L^{p} (\Omega)} + \| u_p (t)\|_{L^1(\partial \Omega)} \Big\}.
\end{array}
$$
The constant $C$ in this inequality in independent of $p\geq N+1$, then we have
\begin{equation} \label{pp}
\begin{array}{l}
\displaystyle 
\|u_p (t) \|^p_{L^\infty (\partial \Omega)} \leq C^p \Big\{ \|\nabla u_p (t) \|^p_{L^{p} (\Omega)} + \| u_p (t) \|^p_{L^1(\partial \Omega)} \Big\}
\\[7pt]
\displaystyle \qquad \qquad \qquad \leq C^p \Big\{ \|\nabla u_p (t) \|^p_{L^{p} (\Omega)} + C^p \Big\}.
\end{array}
\end{equation}
Therefore, 
$$
\int_0^T \|u_p (s) \|^p_{L^\infty (\partial \Omega)} \leq 
\displaystyle C^p \Big\{ \int_0^T \|\nabla u_p (s) \|^p_{L^{p} (\Omega)}  
+ C^p T \Big\}. 
$$
Using \eqref{ooopp} we obtain
$$
\begin{array}{rl}
\displaystyle 
\int_0^T \|u_p (s) \|^p_{L^\infty (\partial \Omega)} \leq  
& \displaystyle C^p \Big\{ \frac12 \int_{\partial \Omega}| u_0|^2 +  C(f) \int_0^T \|u_p (s)\|_{L^\infty (\partial \Omega)}  
+ C^p T \Big\} \\[7pt]
\leq & \displaystyle C^p \|u_0\|_{L^2 (\partial \Omega)}^2 + C^p C(f) \Big( \int_0^T \|u_p (s) \|^p_{L^\infty (\partial \Omega)}  \Big)^{1/p }  T^{1-1/p}
+ C^p T
\\[7pt]
\leq & \displaystyle  {C}^{p^2 / (p-1)} + \frac12 \int_0^T \|u_p  (s) \|^p_{L^\infty (\partial \Omega)}.  
\end{array}
$$
Hence, we obtain
$$ 
\Big( \int_0^T \|u_p (s) \|^p_{L^\infty (\partial \Omega)}  \Big)^{1/p} \leq 
C.
$$
Here the constant $C$ is independent of $p$. 

Then, \eqref{ooopp} implies 
$$
\Big(  \int_0^T \int_\Omega
|\nabla u_p|^{p}  \Big)^{1/p} \leq C^{1/p}.
$$

By an approximation procedure we can use $v=\frac{\partial u}{\partial t}$ as test function to obtain
$$
\int_0^T \int_{\partial \Omega} \Big| \frac{\partial u_p}{\partial t} \Big|^2 + \int_0^T \int_\Omega
\frac{\partial}{\partial t} \frac1p |\nabla u_p|^{p}   =  \int_0^T \int_{\partial \Omega} f \frac{\partial u_p}{\partial t}.
$$
Integrating by parts in time in the last integral, we obtain
\begin{equation} \label{kk}
\begin{array}{l}
\displaystyle 
\int_0^T \int_{\partial \Omega} \Big| \frac{\partial u_p}{\partial t} \Big|^2 +  \int_\Omega
\frac1p |\nabla u_p|^{p} (T)  \\[7pt]
 \displaystyle \qquad = \int_\Omega
\frac1p |\nabla u_0|^{p}  - \int_0^T \int_{\partial \Omega}  \frac{\partial f}{\partial t} u_p 
\\[7pt]
 \displaystyle \qquad \qquad + \int_{\partial \Omega} f(T) u_p (T) - \int_{\partial \Omega} f(0) u_0.
 \end{array} 
\end{equation}
Hence, 
$$
\begin{array}{rl}
\displaystyle
\int_\Omega
|\nabla u_p|^{p} (T) & \displaystyle \leq \int_\Omega
|\nabla u_0|^{p} + p C (f) \int_0^T \|u_p (s)\|_{L^\infty (\partial \Omega)}
 \\[7pt]
& \displaystyle  \qquad
+ p C (f) \|u_p (T)\|_{L^\infty (\partial \Omega)} + pC (f) \|u_0 \|_{L^\infty (\partial \Omega)}
 \\[7pt]
 & \displaystyle 
\leq  \int_\Omega
|\nabla u_0|^{p} + p C + p C \|u_p (T)\|_{L^\infty (\partial \Omega)}.
\end{array}
$$
Now, using \eqref{pp} we get
$$
\begin{array}{rl}
\displaystyle
\|u_p (T) \|^p_{L^\infty (\partial \Omega)} & \displaystyle \leq C^p \Big( \int_\Omega
|\nabla u_0|^{p} + p C + p C \|u_p (T)\|_{L^\infty (\partial \Omega)} \Big) + C^p
\\[7pt]
& \displaystyle  
\leq \frac12 \|u_p (T) \|^p_{L^\infty (\partial \Omega)} + (C^p p C)^{p/(p-1)} + C^p
\end{array} 
$$
and then we conclude that
$$
\|u_p (T) \|_{L^\infty (\partial \Omega)}  \leq C.
$$
As $T$ is any time we obtain
$$
\sup_{t \in [0,T]} \|u_p (t) \|_{L^\infty (\partial \Omega)}  \leq C.
$$

Finally, since $|\nabla u_0| \leq 1$, from \eqref{kk} we conclude that
$$
\int_0^T \int_{\partial \Omega} \Big|\frac{\partial u_p}{\partial t} \Big|^2 \leq C
$$
This ends the proof.
\end{proof}

Now, we prove that the functionals $E_p$ converge in the sense of Mosco to the limit functional $E_\infty$.

\begin{proof}[Proof of Theorem \ref{teo.Mosco}]
First, we want to show that \eqref{e2Mos.prelim} holds, that is,
\begin{equation}\label{e2Mos.sec}
\forall \, u \in D(E_\infty) \ \ \exists u_p \in D(E_p) \ : \ u_p \to u \ \ \hbox{and} \ \
E (u) \geq \limsup_{n \to \infty} E_p(u_p).
\end{equation}
Given $u \in D(E_\infty)$, that is, $u \in A_\infty$, we just take
$$
u_p \equiv u
$$
as the desired sequence. 
We clearly have $u_p \to u$ strongly in $L^2 (\partial \Omega)$. 

Now, from the fact that $u\in A_\infty$ we have that there exists
$v^* :\overline{\Omega} \mapsto \mathbb{R}$ with $ |\nabla v^*|\leq 1 \mbox{ a.e } \Omega$ and
$v^*|_{\partial \Omega} = u$. Hence, we obtain that $u \in D(E_p)$, that is, $u\in trace (W^{1,p} (\Omega)$
and 
$$
\begin{array}{rl}
E_p(u_p)  & \displaystyle =  \min_{v \in W^{1,p} (\Omega), trace (v)=u} \frac1p \int_\Omega |\nabla v|^p \\[7pt]
 & \displaystyle \leq  \frac1p \int_\Omega |\nabla v^*|^p \\[7pt]
 & \displaystyle \leq  \frac1p |\Omega | \to 0 \qquad \mbox{as } p \to \infty.
 \end{array}
$$
Then we have, 
$$
0=E (u) \geq \limsup_{n \to \infty} E_p(u_p) =0
$$
as we wanted to show.

Now, we have to prove that \eqref{e3Mos.prelim} also holds, namely,
\begin{equation}\label{e3Mos.sec}
\hbox{for every subsequence} \ \{p_k\}, \ E_\infty (u) \leq \liminf_{k}
E_{p_k}(u_k) \ \hbox{whenever} \ u_{k}\rightharpoonup u.
\end{equation}
To see this, we first observe that when $u \in A_\infty = D(E_\infty)$ we have 
$E_\infty (u)=0$ and we trivially obtain $E_\infty (u) \leq \liminf_{k}
E_{p_k}(u_k)$ since $E_{p_k} (u_k) \geq 0$. 

Also, we can assume that $\liminf_{k}
E_{p_k}(u_k) < +\infty$ (otherwise the desired inequality holds trivially). 
Hence, for a subsequence we have that there is a constant $C$ such that
$$
\frac1{p_k} \min_{v \in W^{1,p_k} (\Omega), trace (v)=u_k}  \int_\Omega |\nabla v|^{p_k} \leq C.
$$
Call $v_k$ a function in $W^{1,p_k} (\Omega)$ that attains the minimum. For this $v_k$ we have 
$$
\left( \int_\Omega |\nabla v_k|^{p_k} \right)^{1/p_k} \leq ( p_k C )^{1/p_k}.
$$
Now, for $2<q<\infty$, we obtain
$$
\begin{array}{rl}
\displaystyle 
\left( \int_\Omega |\nabla v_k|^{q} \right)^{1/q} & \displaystyle \leq 
|\Omega|^{(p_k-q)/p_k q} \left( \int_\Omega |\nabla v_k|^{p_k} \right)^{1/p_k} \\[7pt]
& \displaystyle  \leq |\Omega|^{(p_k-q)/p_k q}
( p_k C )^{1/p_k}.
\end{array}
$$
The right hand side is bounded and hence we can take the limit as $p_k\to \infty$ to obtain
that $v_k \rightharpoonup v^*$ weakly in $W^{1,q} (\Omega)$. This limit $v^*$ verifies
$$
\left( \int_\Omega |\nabla v^*|^{q} \right)^{1/q}  
\leq |\Omega|^{1/q}.
$$
Hence, taking $q\to \infty$ we conclude that $v^* \in W^{1,\infty} (\Omega)$
and 
$$
|\nabla v^*| \leq 1, \qquad \mbox{a.e } \Omega.
$$
Now, from the weak convergence of $v_k$ to $v^*$ in $W^{1,q} (\Omega)$
using the Sobolev trace embedding we get that $u_k = trace (v_k) \to u=trace(v^*)$ strongly in $L^2 (\partial \Omega)$
and hence we have that $u \in A_\infty = D(E_\infty)$. Then, we have  $$0= E_\infty (u) \leq \liminf_{k}
E_{p_k}(u_k)$$ since $E_{p_k} (u_k) \geq 0$, as we wanted to show.
\end{proof}

As a consequence we obtain the convergence of the corresponding solutions to the associated evolution problems.

\begin{proof}[Proof of Theorem \ref{corol.converg}]
We can apply Theorem \ref{convergencia1.Mosco} to obtain the first part of the result, namely, 
$$
u_p \to u_\infty
$$
as $p\to \infty$ in $C([0,T]: L^2(\partial \Omega))$ and the limit $u_\infty$ is characterized as the solution to 
the limit problem
\eqref{eq-limite}.

To complete the proof we observe that, from the uniform bounds obtained in Lemma \ref{lema.cotas}, we obtain the existence of a subsequence $p_i \to \infty$ such that
the convergences stated in \eqref{convergencias} hold. 
\end{proof}

\section{Mass transport interpretation of the limit problem} \label{sect-Mass}

We relate the limit problem with an optimal mass transport problem 
with a cost given by the distance between points inside $\Omega$
that is defined as the infimum of the lengths of curves going from $x$ to $y$, that is,
$$
d_\Omega (x,y) = \inf_{\gamma (0) = x, \gamma (1) =y} \mbox{lenght} (\gamma (t)).
$$
When the domain $\Omega$ is convex the distance $d_\Omega$ coincides with the Euclidean distance, we have $d_\Omega (x,y) = |x-y|$.

Given two measures $\mu$, $\nu$ on $\partial \Omega$ with the same total mass we consider
the transport cost (Monge-Kantorovich mass transport problem)
$$
C(\mu,\nu) = \min_{\theta (x,y): \theta|_x = \mu, \theta|_y = \nu} \int_{\partial \Omega \times \partial \Omega} 
d_{\Omega} (x,y) d\theta (x,y).
$$
Here by $\theta|_x$ we denote the first marginal of $\theta$, that is, $\theta|_x (E) = \theta (E \times \partial \Omega)$
(and similarly with $\theta|_y$ we denote the second marginal of $\theta$).

Associated with an optimal mass transport problem we have its dual formulation that is given by
$$
C(\mu,\nu) = \max_{v: | v(x)-v(y)| \leq d_\Omega (x,y)} \int_{\partial \Omega } 
v(x) (d\mu(x) - d\nu(x)).
$$
Maximizers of the dual problem are called Kantorovich potentials for the optimal mass transport problem. 

It turns out that the limit of the solutions, $u_\infty(\cdot,t)$, is a Kantorovich potential for
the optimal mass transport problem between $f(\cdot,t)$ and $\frac{\partial u_\infty}{\partial t} (\cdot,t)$. 

\begin{proof}[Proof of Theorem \ref{teo-mass}]
First, let us prove that the limit function $u_\infty$ is admissible for the dual problem.
Given two points $x,y\in \partial \Omega$, using that $|\nabla u_\infty(\cdot,t)| \leq 1$ a.e. $\Omega$, we have 
$$
\begin{array}{rl}
\displaystyle 
|u_\infty(x,t) - u_\infty(y,t)| & \displaystyle = \left| \int_0^1 \frac{\partial u_\infty(\gamma(s) ,t)}{\partial s} (s) ds \right|
\\[7pt]
& \displaystyle = \left| \int_0^1 \langle \nabla u_\infty (\gamma(s),t) , \gamma'(s) \rangle ds
\right| \\[7pt]
& \displaystyle  \leq 
\mbox{lenght} (\gamma (s)) 
\end{array}
$$
and hence we obtain
$$
|u_\infty(x) - u_\infty(y)| \leq d_\Omega (x,y).
$$

Now, we show that in fact $u_\infty(\cdot,t)$ is a solution to the dual problem. We have that
$u_\infty(x,t)$ solves the limit equation,
$$
f(x,t)- \frac{\partial u_\infty}{\partial t} (x,t) \in \partial E_\infty (u(x,t))  
$$
that is,
$$
E_\infty (v(x)) \geq E_\infty (u_\infty(x,t)) + \int_{\partial \Omega} \Big( f(x,t)- \frac{\partial u_\infty}{\partial t} (x,t) \Big)
(v(x)-u_\infty (x,t)) 
$$
Take $v \in A_\infty $. Since $u_\infty(\cdot,t)\in A_\infty$ we have 
$$
0 \geq  \int_{\partial \Omega} \Big( f(x,t)- \frac{\partial u_\infty}{\partial t} (x,t) \Big)
(v(x)-u_\infty(x,t)) 
$$
and therefore, 
$$
\int_{\partial \Omega } 
u_\infty(x,t) \Big(\frac{\partial u_\infty}{\partial t}(x,t) - f(x,t) \Big) \geq 
\int_{\partial \Omega } 
v(x) \Big(\frac{\partial u_\infty}{\partial t}(x,t) - f(x,t) \Big),
$$
for every $v$ such that $| v(x)-v(y)| \leq d_\Omega (x,y)$.

We have obtained that $u_\infty(\cdot,t)$ is a Kantorovich potential for the optimal mass transport problem 
between $f(\cdot,t) d\sigma $ and $\frac{\partial u_\infty}{\partial t} (\cdot,t) d\sigma $.
\end{proof}

\section{Examples} \label{sect-examples}

In this final section we include some simple examples in which one can find the solution to the limit evolution problem,
\begin{equation}\label{eq-limite.examples}
\left\{
\begin{array}{ll}
\displaystyle f(x,t)- \frac{\partial u_\infty}{\partial t} (x,t) \in \partial E_\infty (u_\infty(x,t))  \qquad & x\in \partial \Omega, t>0, \\[7pt]
u(x,0)=u_0(x) \qquad & x \in \partial \Omega.
\end{array}
\right.
\end{equation}

{\bf Example 1.} First, let us deal with the $1-$dimensional case and consider $\Omega =(0,1)$, take 
$$
f(x,t) = 
\left\{ 
\begin{array}{ll}
0, \qquad & x=0, t>0,\\[7pt]
1, \qquad & x=1, t>0,
\end{array}
\right.
$$
(notice that $f$ is defined on $\partial \Omega \times (0,T)$) and
$$
u_0 \equiv 0.
$$

Then we have
$$
u_\infty(x,t) = 
\left\{ 
\begin{array}{ll}
0, \qquad  & x=0, 0\leq t \leq 1,\\[7pt]
t, \qquad & x=1, 0\leq t \leq 1,
\end{array}
\right.
$$
and 
$$
u_\infty(x,t) = 
\left\{ 
\begin{array}{ll}
\displaystyle \frac12 (t-1)  , \qquad & x=0, 1\leq t ,\\[7pt]
\displaystyle \frac12 (t-1) +1 , \qquad & x=1, 1\leq t .
\end{array}
\right.
$$

Notice that we have 
$$
|u_\infty(1,t) -u_\infty(0,t)| \leq 1 = d_\Omega (0,1)=1, \qquad \mbox{ for every } t \geq 0.
$$

Also remark that the solution starts to grow at $x=1$ with $\frac{\partial u_\infty}{\partial t} (1,t) = 1$
until it reaches $u_\infty(1,t_0)=1$ (this happens at $t_0=1$) and next it grows at the slower rate 
$\frac{\partial u_\infty}{\partial t} (1,t) = 1/2$ (but also grows at $x=0$ with $\frac{\partial u_\infty}{\partial t} (0,t) = 1/2$). This
is due to the fact that the unit mass added at $x=1$ is divided between two locations $x=0$ and $x=1$ in order to 
keep the constraint $|u_\infty(1,t) -u_\infty(0,t)| \leq 1$ for times $t \geq 1$.

\medskip

{\bf Example 2.} We can also consider a nontrivial initial condition for the setting considered in the previous example.

Let $\Omega =(0,1)$. Take, as before, 
$$
f(x,t) = 
\left\{ 
\begin{array}{ll}
0, \qquad & x=0, t>0,\\[7pt]
1, \qquad & x=1, t>0,
\end{array}
\right.
$$
(notice that $f$ is defined on $\partial \Omega \times (0,T)$) and fix a nonnegative $C^1$ initial condition $u_0$
with $|u_0' (x) | \leq 1$ for $x \in [0,1]$. 

Then we have
$$
u_\infty(x,t) = 
\left\{ 
\begin{array}{ll}
u_0(0), \qquad & x=0, 0\leq t \leq t_0,\\[7pt]
u_0(1) + t, \qquad & x=1, 0\leq t \leq t_0,
\end{array}
\right.
$$
with $t_0$ the first time at which $u_0(1)+t_0 - u_0(0) = 1$, that is
$$
t_0 = u_0(0) -u_0(1) +1. 
$$
Notice that $t_0 \geq 0$ due to the fact that $u_0(0) -u_0(1) +1 = u_0' (\xi) +1 \geq 0$. 
Also notice that $u_\infty(x,t) \in A_\infty$, since there exists a function $v$ with $|v'|\leq 1$ in $[0,1]$ such that $v(0)= u_0(0)$,
$v(1) = u_0(1) + t$ (in addition, this function $v$ can be chosen satisfying $v\geq u_0$ in $[0,1]$).

For times larger than $t_0$ we have
$$
u_\infty (x,t) = 
\left\{ 
\begin{array}{ll}
\displaystyle  u_0(0) + \frac12 (t-t_0)  , \qquad & x=0, t_0\leq t ,\\[7pt]
\displaystyle u_0(1) + \frac12 (t-t_0) +t_0 , \qquad & x=1, t_0 \leq t .
\end{array}
\right.
$$

\medskip

{\bf Example 3.} Now, we extend these ideas to several dimensions. 
Take a fixed domain $\Omega \subset \mathbb{R}^N$, fix a subdomain of its boundary 
$\Gamma \subset \partial \Omega$ and consider $f:\partial \Omega \times (0,T) \mapsto \mathbb{R}^N$,
$$
f(x,t) = \chi_{\Gamma} (x)
$$
and, as before, 
$$
u_0 \equiv 0. 
$$

Remark that our previous example, Example 1, is a particular case of this more general setting.

In this case the solution $u_\infty (x,t)$ to the limit problem is given by
$$
u_\infty (x,t) = (a(t) - d_\Omega (x,\Gamma))_+,
$$
with $a(t)$ the solution to the ODE
$$
\left\{
\begin{array}{l}
a'(t) \Big|\{ x\in \partial \Omega : d_\Omega (x, \Gamma) < a(t)\} \Big|_{H^{N-1}} = |\Gamma|_{H^{N-1}}, \\[7pt]
a(0)=0.
\end{array}
\right.
$$
Here we denoted by $| E |_{H^{N-1}}$ the $N-1$-dimensional surface measure of a measurable set $E \subset \partial \Omega$. 

Notice that the support of $u_\infty (\cdot, t)$ in $\partial \Omega$ can be disconnected even if
the domain is strictly convex and the set where the source is localized $\Gamma$ is connected.
In fact, this is the case when the set 
$$
\Big\{ x\in \partial \Omega : d_\Omega (x, \Gamma) < k \Big\}
$$
is disconnected for some $k>0$.  Also notice that, since $\Omega$ is bounded and $\partial \Omega$ is
smooth (it has finite $H^{N-1}-$measure), there exists a finite time $t_0$ such that the support of $u_\infty (\cdot,t)$ is the whole
$\partial \Omega$ for times $t\geq t_0$.
At this time $t_0$ we have 
$$
a(t_0) = \max_{x\in \partial \Omega} d_\Omega (x, \Gamma)
$$
and then we have
$$
u_\infty (x,t) = \Big(\max_{x\in \partial \Omega} d_\Omega (x, \Gamma) - dist(x,\Gamma) \Big),
$$
After this time the solution is given by 
$$
u_\infty (x,t) = \Big( \frac{|\Gamma|_{H^{N-1}}}{|\partial \Omega|_{H^{N-1}}} t + \max_{x\in \partial \Omega} d_\Omega (x, \Gamma)- dist(x,\Gamma)\Big),
$$
That is, after $t_0$ the solution grows uniformly in the whole $\partial \Omega$ with speed $\frac{|\Gamma|_{H^{N-1}}}{|\partial \Omega|_{H^{N-1}}} $.

\medskip

{\bf Acknowledgements.} J.D.R. is partially supported by 
CONICET grant PIP GI No 11220150100036CO
(Argentina), PICT-2018-03183 (Argentina) and UBACyT grant 20020160100155BA (Argentina).

\bigskip

\bigskip

%----------------------------------------------------------------------------------------
%	BIBLIOGRAPHY
%----------------------------------------------------------------------------------------

\end{document}